\documentclass[12pt,reqno]{amsart}
\usepackage{a4wide}

\numberwithin{equation}{section}

\newtheorem{theorem}{Theorem}[section]
\newtheorem{proposition}[theorem]{Proposition}

\newtheorem{lemma}[theorem]{Lemma}

\theoremstyle{definition}

\newtheorem{remark}[theorem]{Remark}

\newcommand{\R}{\mathbb{R}}

\begin{document}

\title[the Non-degeneracy
and existence of solutions
 ]
{   Non-degeneracy
and existence of new solutions for  the   Schr\"odinger equations
 }

 \author{Yuxia Guo, Monica Musso, Shuangjie Peng and Shusen Yan
}

\address{Department of  Mathematics, Tsinghua University, Beijing, 100084, P.R.China}

\email{yguo@mail.tsinghua.edu.cn   }

\address{Department of Mathematical Sciences University of Bath, Bath BA2 7AY, United
Kingdom}

\email{m.musso@bath.ac.uk}

\address{ School of Mathematics and  Statistics, Central China Normal University, Wuhan, P.R. China}

\email{ sjpeng@mail.ccnu.edu.cn}

\address{ School of Mathematics and  Statistics, Central China Normal University, Wuhan, P.R. China}

\email{ syan@mail.ccnu.edu.cn}

\begin{abstract}

We  consider the following  nonlinear problem
\begin{equation}
\label{eq}
- \Delta u + V(|y|)u=u^{p},\quad u>0 \ \   \mbox{in} \  \R^N, \ \ \
 u \in H^1(\R^N),
\end{equation}
where $V(r)$ is a  positive function, $1<p <\frac{N+2}{N-2}$. We
show that the multi-bump solutions constructed in \cite{WY} is non-degenerate
in a suitable symmetric space. We also use this non-degenerate result
to construct new solutions for \eqref{eq}.

\end{abstract}

\maketitle

\section{Introduction}

Consider  the following nonlinear elliptic problem

\begin{equation}\label{equation}
\begin{cases}
-\Delta u + V(y) u = u^p, \quad u>0, & \text{in}\; \mathbb R^N,\\
 \lim_{|y|\to+\infty}
u(y)=0,
\end{cases}
\end{equation}
where $N\ge 2$, $p\in \bigl(1, \frac{N+2}{N-2}\bigr)$, and $V$ is a continuous function which satisfies

\[
 \lim_{|y|\to+\infty} V(y)=1.
 \]
 The existence of nontrivial  solutions for \eqref{equation}, or the following nonlinear field equations in subcritical case

 \begin{equation}\label{equation1}
\begin{cases}
-\Delta u +  u = Q(x)u^p, \quad u>0, & \text{in}\; \mathbb R^N,\\
 \lim_{|y|\to+\infty}
u(y)=0,
\end{cases}
\end{equation}
 attracts a lot
of attentions in the last four decades. The readers can refer to \cite{BL}-\cite{CDS}, \cite{DN,Lion} and the references therein.
On the other hand, for results on the singularly perturbed problems corresponding to \eqref{equation} and  \eqref{equation1},  we refer
the readers to
\cite{ABC}--\cite{AMN2}, \cite{CNY1,CNY2},\cite{DY}--\cite{DF4}, \cite{NY,R}.

 The functional corresponding to \eqref{equation} is given by

\[
I(u)=\frac12\int_{\mathbb R^N} \bigl(|\nabla u|^2 +V(y) u^2\bigr)-\frac1{p+1}\int_{\mathbb R^N} |u|^{p+1},\quad u\in H^1(\mathbb R^N).
\]

 It is
well-known   that the following problem has a unique solution  $U$

\begin{equation}\label{10-18-4}
\begin{cases}
-\Delta u +u= u^{p},\;\;u>0, &\text{in}\; \R^N,\\
u(y)\to 0, &\text{as}\; |y|\to \infty.
\end{cases}
\end{equation}
satisfying $U(y)=U(|y|)$, $U'<0$.  We denote   $U_{x_j}(y)=U(y-x_j)$.

For any  integer $k>0$ and $R>0$ large, we define

\[
D_{k, R}=\bigl\{ (x_1,\cdots, x_k):\;  x_j\in \mathbb R^N, \; |x_j|\ge R, j=1, \cdots, k,  \, |x_i-x_j|\ge R, i\ne j\bigr\}.
\]
Then for any $(x_1,\cdots, x_k)\in D_{k, R}$, the  function
$
\sum_{j=1}^k U_{x_j}
$
is an approximate solution of \eqref{equation}. A natural question is whether we can
make a small correction for this approximate solution to obtain a true solution for \eqref{equation}.

Direct computations give

\begin{equation}\label{1-23-4}
I\bigl( \sum_{j=1}^k U_{x_j}\bigr)\approx k A + B_1 \sum_{j=1}^k\bigl( V(x_j)-1\bigr) - B_2\sum_{i\ne j} U(|x_i-x_j|),
\end{equation}
where $B_1$ and $B_2$ are some positive constants, and  $A= I(U)$.

To find a stable critical point for the function on the right hand side of \eqref{1-23-4} in $D_{k, R}$, the main difficulty is
that the terms in this function may be of different order, depending on the location of each $x_j$. To avoid this difficulty, in
\cite{WY},  it assume that
  $V(y)$ is radial, and  the following problem is studied

\begin{equation}
\label{1.4}
 -\Delta u +V(|y|)u = u^p, u >0 \ \  \mbox{in}\  \R^N, \ \
  u \in H^1(\R^N).
\end{equation}

We can see easily  that if the function on the right hand side of \eqref{1-23-4} has a critical point in $D_{k, R}$, then $V(y)-1$ can not be
negative near the infinity. From this observation, in \cite{WY}, the following condition is imposed.

(V):  There  are constants $a>0$, $\alpha>1$,  and $\gamma>0$,
such that

\begin{equation}\label{V}
V(r)= 1+\frac a {r^\alpha} +O\bigl(\frac1{r^{\alpha+\gamma}}\bigr),
\end{equation}
as $r\to +\infty$.

Let

\[
x_j=\bigl(r \cos\frac{2(j-1)\pi}k, r\sin\frac{2(j-1)\pi}k,0\bigr),\quad j=1,\cdots,k,
\]
where $0$ is the zero vector in $\R^{N-2}$.  Then, it is easy to calculate

\begin{equation}\label{1-24-4}
I\bigl( \sum_{j=1}^k U_{x_j}\bigr)\approx k A + \frac{ a B_1 k }{ r^\alpha}  - B_2 kU(|x_2-x_1|).
\end{equation}
 For $k>0$ large, the function $\frac{ a B_1  }{ r^\alpha}  - B_2 U(|x_2-x_1|)$  has a local maximum point
$r_k\in [r_0 k \ln k, r_1
 k \ln k]$, where    $r_1>r_0>0$  are some constants.  Therefore, \eqref{1.4} has a solution with $k$ bumps near infinity provided $k$ is large enough. To
 be more precisely, we set
 $y=(y',y'')$, $y'\in \R^2$, $y''\in \R^{N-2}$.  Define
\[
\begin{split}
H_{s}=\bigl\{ u: &   u\;\text{is even in} \;y_2,\cdots, y_N,\\
& u(r\cos\theta , r\sin\theta, y'')=
u(r\cos(\theta+\frac{2\pi j}k) , r\sin(\theta+\frac{2\pi j}k), y'')
\bigr\}.
\end{split}
\]

Let
\[
W_r(y)=\sum_{j=1}^k U_{x_j}(y).
\]

The following result is proved in \cite{WY}.

{\bf Theorem~A.}   {\it Suppose that $V(r)$ satisfies \eqref{V}. Then there is an
integer $k_0>0$, such that for any integer  $k\ge k_0$, \eqref{1.4}
has a solution $u_k$ of the form
\[
u_k = W_{r_k}(y)+\omega_k,
\]
where  $\omega_k\in H_s\cap  H^1(\R^N)$,   $r_k \in [r_0 k \ln k, r_1
 k \ln k]$ and as $k\to +\infty$,
\[
\int_{\R^N} \bigl(|D \omega_k |^2+\omega_k^2\bigr)\to 0.
\]
}

 If $N\ge 4$,  for any large integer $n>0$, we let

 \[
p_j=\Bigl(0,0, t \cos\frac{2(j-1)\pi}n, t\sin\frac{2(j-1)\pi}n,0,\cdots,0\Bigr),\quad j=1,\cdots,n.
\]

  We are now interested in finding a new solution to \eqref{1.4}
whose shape is, at main order,
\begin{equation}\label{86-3-4}
u\approx \sum_{j=1}^k U_{x_j} +\sum_{j=1}^n U_{p_j}.
\end{equation}

We calculate

\begin{equation}\label{88-3-4}
\begin{split}
&I\Bigl( \sum_{j=1}^k U_{x_j} +\sum_{j=1}^n U_{p_j}  \Bigr)\\
 \approx &k  A + \frac{ a B_1 k }{ r^\alpha}  - B_2 kU(|x_2-x_1|)\\
  &+nA + \frac{ a B_1 n }{ t^\alpha}  - B_2 kU(|p_2-p_1|).
  \end{split}
\end{equation}
 Similar to \eqref{1-23-4}, if $n>>k$, the terms on the right side hand of \eqref{88-3-4} are of
 different order. In other words, it is hard to see the contribution to the energy  from the bumps $U_{p_j}$.
 Therefore, it is very difficult to use a reduction
 argument directly to construct solutions of the form \eqref{86-3-4}.

 Following the idea in \cite{GMPY}, for any fixed large integer $k>0$, we will use  $u_k +\sum_{j=1}^n U_{p_j}$ as an approximate
 solution for \eqref{1.4}. The key point that we can make a small correction for $u_k +\sum_{j=1}^n U_{p_j}$  to obtain a true solution
 for \eqref{1.4} is to prove
  the non-degeneracy of the solution $u_k$ in $H_s\cap  H^1(\R^N)$ in the sense that the
following linearized operator

 \begin{equation}\label{2-22-12}
 L_k \xi= -\Delta \xi + V(|y|) \xi -p u_k^{p-1}\xi,
 \end{equation}
 has trivial kernel in  $H_s\cap  H^1(\R^N)$.
 To prove such non-degeneracy result, we need to  impose the following conditions on $V$

\begin{equation}\label{V1}
V(r)= 1+\frac {a_1} {r^\alpha} +\frac {a_2} {r^{\alpha+1}} +O\bigl(\frac1{r^{\alpha+2}}\bigr),
\end{equation}

\begin{equation}\label{V2}
V'(r)= -\frac {a_1\alpha} {r^{\alpha+1}} -\frac {a_2(\alpha+1)} {r^{\alpha+2}} +O\bigl(\frac1{r^{\alpha+3}}\bigr),
\end{equation}
where  $\alpha>\max\bigl(\frac 4{p-1}, 2\bigr)$, $a_1>0$  and $a_2$ are some   constants.

The main result of this paper is the following.

 \begin{theorem}\label{th11}

 Suppose that $V(r)$ satisfies \eqref{V1}  and \eqref{V2}. Let $\xi\in H_s\cap H^1(\mathbb R^N)$ be a solution of $L_k \xi=0$. Then $\xi=0$.

 \end{theorem}

 Similar problem is studied for the following prescribed scalar curvature equation

 \begin{equation}\label{1-13-5}
- \Delta u =K(y)u^{2^*-1},\quad u>0 \ \ \   \mbox{in} \  \R^N, \ \
 u \in D^{1, 2}(\R^N).
\end{equation}
where $2^* = {2N \over N-2}$, under the condition that $K(r)$ has a non-degenerate local maximum point $r_0>0$. There are
some technical differences in the study of the non-degeneracy of solutions for the nonlinear Schr\"odinger equations and
the prescribed scalar curvature equation. For the prescribed scalar curvature equation, the non-degeneracy of the local maximum point $r_0>0$
plays an essential role in using the local Pohozaev identities to kill the possible nontrivial kernel of the linear operator.
For the nonlinear Schr\"odinger equations, we can regard  condition \eqref{V2} as $V$ has a critical point at infinity. But it is
not clear that such an expansion implies certain non-degeneracy of $V$ at infinity. This difference, together with the exponential decay
of the solution $U$ of \eqref{10-18-4}, leads us to use the local Pohozaev identities in quite different ways.

A direct consequence of Theorem~\ref{th11} is the following result for the existence   of new solutions  for \eqref{1.4}.

 \begin{theorem}\label{th12}

 Suppose that $V(r)$ satisfies the assumptions in Theorem~\ref{th11} and $N\ge 4$. Let $u_k$ be a solution in Theorem~A and $k>0$ is a large even number. Then there is an
integer $n_0>0$, depending on $k$, such that for any even number  $n\ge n_0$,  \eqref{1.4}
has a solution of the form \eqref{86-3-4} for some $t_n\to +\infty$.

 \end{theorem}

 This paper is organized as follows. In section~2, we will revisit the proof of Theorem~A in order to obtains some
 extra estimates needed in the proof of Theorem~\ref{th11}. The main result on the non-degeneracy of the solutions will be
 proved in section~3, while the construction of the new solutions will be carried out in section~4.

 \section{ revisit the existence problem}

In this section, we will briefly revisit the existence problem for \eqref{1.4} and give a different proof of Theorem~A, so that
we can obtained extra estimates which are needed in the proof of the non-degeneracy result.

We denote
 \[
Z_j=\frac{\partial U_{x_j}}{\partial r},\quad j=1,\cdots,k,
  \]
where $x_j=\bigl(r \cos\frac{2(j-1)\pi}k,
r\sin\frac{2(j-1)\pi}k,0\bigr) $.

Let

\[
\begin{split}
\tilde H_{s}=\bigl\{ u: &   u\;\text{is even in} \;y_2;\;  u(y', y'')= u(y', |y''|)\\
& u(r\cos\theta , r\sin\theta, y'')=
u(r\cos(\theta+\frac{2\pi j}k) , r\sin(\theta+\frac{2\pi j}k), y'')
\bigr\},
\end{split}
\]

\[
\tilde E_k=\bigl\{ v:  v\in \tilde H_s\cap C(\mathbb R^N),\; \sum_{j=1}^k \int_{\R^N} U_{x_j}^{p-1} Z_j v=0, \;j=1,\cdots,k\bigr\},
\]
and

\begin{equation}\label{1-20-4}
 S_k= \bigl[ (\frac \alpha{2\pi}-\beta)k\ln k, (\frac \alpha{2\pi}+\beta) k\ln k\bigr],
\end{equation}
where $\alpha$ is the constant in the expansion for $V$, and  $\beta>0$
is a small constant.

In this paper, we always assume that $r\in S_k$.

Take a small constant $\tau >0$.  Define the norm

\[
\| u\|_*= \sup_{y\in \mathbb R^N} \frac{|u(y)|}{\sum_{j=1}^k e^{-\tau |y-x_j|}}.
\]

Let

\[
 Lv= -\Delta v +V(|y|) v- p W_r^{p-1} v.
\]

  For  $f\in \tilde H_s\cap C(\mathbb R^N)$,  we consider the following problem

\begin{equation}\label{1-25-4}
Lv = f+  b_k \sum_{j=1}^k  U_{x_j}^{p-1} Z_j
\end{equation}
for some  $
v\in \tilde E_k
$ and $a_k\in \mathbb R$.
We have

\begin{lemma}\label{l21}
 Assume that $(a_k, v_k)$ solves \eqref{1-25-4} for $f= f_k$. If $\| f_k\|_* \to 0$  as $k\to +\infty$,
then $\| v_k\|_* \to 0$  as $k\to +\infty$.

\end{lemma}

\begin{proof}

We argue by contradiction.  Suppose that there are $k_m\to +\infty$,
$r_{k_m}\in  S_{k_m}$, and $(b_{k_m}, v_{k_m})$ solving \eqref{1-25-4}, such that
$\|f_{k_m}\|_*\to 0$,  $\|v_{k_m}\|_*=1$.  For simplicity of the notation, we drop the subscript $k_m$.

We have

\begin{equation}\label{2-l21}
v(x) =\int_{\mathbb R^N} G(y, x) \bigl(p W_r^{p-1} v +f + b  \sum_{j=1}^k  U_{x_j}^{p-1} Z_j\bigr),
\end{equation}
where $G(y, x)$ is the Green's function of $-\Delta + V(|y|)$ in $\mathbb R^N$.

It is easy to prove

\[
\bigl| \int_{\mathbb R^N} G(y, x) f\bigr|\le C\|f\|_* \sum_{j=1}^k  \int_{\mathbb R^N} G(y, x) e^{-\tau |y-x_j|}\,dy\le C
\|f\|_* \sum_{j=1}^k   e^{-\tau |x-x_j|},
\]
 and for $\tau>0$ small,

\[
\begin{split}
\bigl| \int_{\mathbb R^N} G(y, x) W_r^{p-1} v\bigr|\le & C \|v\|_*  \int_{\mathbb R^N} G(y, x) \bigl( \sum_{j=1}^k  e^{-|y-x_j|}\bigr)^{p-1}\bigl( \sum_{j=1}^k  e^{-\tau|y-x_j|}\bigr)\,dy\\
\le & C
\|v\|_* \sum_{j=1}^k   e^{- 2\tau |x-x_j|}.
\end{split}
\]
Moreover,

\[
\bigl| \int_{\mathbb R^N} G(y, x)  \sum_{j=1}^k  U_{x_j}^{p-1} Z_j
\bigr|\le C \sum_{j=1}^k   e^{- |x-x_j|}.
\]

On the other hand, from

\[
b\int_{\mathbb R^N} \sum_{j=1}^k  U_{x_j}^{p-1} Z_j Z_1  = \int_{\mathbb R^N} \bigl( L v - f\bigr) Z_1,
\]
we can prove that $b\to 0$.
So, it holds

\begin{equation}\label{100-27-2}
|v(x)|\bigl( \sum_{j=1}^k  e^{-\tau|y-x_j|}\bigr)^{-1} \le o(1) + C\|v\|_* \bigl(\sum_{j=1}^k e^{-\tau|y-x_j|}\bigr)^{-1}\sum_{j=1}^k   e^{- 2\tau |x-x_j|}.
\end{equation}

Since $v\in \tilde  E$, we can prove $\| v\|_{L^\infty(B_R(x_j))}\to 0$  for any $R>0$.  So we find from \eqref{100-27-2} that $\|v\|_*\to 0$. This is a contradiction.

\end{proof}

Now we want to  solve the following problem

\begin{equation}\label{10-25-4}
L \omega =  l_k  +  R(\omega)+ b_k \sum_{j=1}^k  U_{x_j}^{p-1} Z_j,
\end{equation}
for $\omega \in E_k$ and $b_k\in \mathbb R$, where

\[
l_k=\sum_{i=1}^k (V(|y|)-1) U_{x_i} +
\bigl( W_r^{p}-\sum_{i=1}^k U_{x_i}^p\bigr),
\]
and

\[
R(\omega)=(W_r+\omega)^p  -W_r^p  -p
W_p^{p-1} \omega.
\]

 We have

\begin{lemma}\label{100l1-25-3}
There is a small $\sigma>0$, such that

\[
\|l_k\|_* \le \frac{C }{r^\alpha}+ Ce^{- \min[\frac p2-\tau, 1]|x_2-x_1|},
\]
where $r=|x_i|$.

\end{lemma}

\begin{proof}

 First, it holds

 \[
 \bigl| \sum_{i=1}^k (V(|y|)-1) U_{x_i}\bigr|\le  \frac{C}{1+  |y|^\alpha} \sum_{i=1}^k e^{-|y-x_i|}\le  \frac{C}{   r^\alpha} \sum_{i=1}^k e^{-\tau |y-x_i|}.
 \]

We define

\[
\Omega_j=\Bigl\{ y=(y',y'')\in\R^2\times \R^{N-2}:
 \Bigl\langle \frac {(y', 0)}{|y'|}, \frac{x_{k, j}}{|x_{k, j}|}\Bigr\rangle\ge \cos \frac{\pi}{k}\Bigr\}.
\]

 For any $y\in \Omega_1$,  we have $|y-x_j|\ge |y-x_1|$, $j=2,\cdots, k$.
Thus, for any $y\in \Omega_1$,  if $p>2$, it holds

\[
\begin{split}
\bigl| W^{p}-\sum_{i=1}^k U_{x_i}^p\bigr|\le & C \sum_{i=2}^k U_{x_1}^{p-1} U_{x_i}\le C e^{-(p-2) |y-x_1|}  \sum_{i=2}^k e^{- |x_j-x_1|}\\
\le & C e^{-\tau |y-x_1|}  \sum_{i=2}^k e^{- |x_j-x_1|},
\end{split}
\]
while if $p\in (1, 2]$,

\[
\begin{split}
&\bigl| W^{p}-\sum_{i=1}^k U_{x_i}^p\bigr| \le  C \sum_{i=2}^k U_{x_1}^{\frac p2} U_{x_i}^{\frac p2}\\
\le &  C e^{-\tau  |y-x_1|}  \sum_{i=2}^k e^{-(\frac p2-\tau)|x_j-x_1|}.
\end{split}
\]

\end{proof}

Using Lemmas~\ref{l21} and the contraction mapping theorem,  we can  prove that
there is an
integer $k_0>0$, such that for any integer  $k\ge k_0$, \eqref{10-25-4}
has a solution $\omega_k\in \tilde  E_k$. Moreover,

\begin{equation}\label{11-25-4}
\|\omega_k\|_*\le C \|l_k\|_* \le \frac{C }{r^\alpha}+C e^{- \min[\frac p2-\tau, 1]|x_2-x_1|}.
\end{equation}

Now we can calculate

\begin{equation}\label{12-25-4}
I\bigl( W_{r}+\omega_k\bigr)=  k A + \frac{ a B_1 k }{ r^\alpha}  - k \bigl(B_2+o(1)\bigr) U(|x_2-x_1|)+ k O\bigl( \frac{ 1 }{ r^{\alpha+\gamma}}+U^{1+\sigma}(|x_2-x_1|)\bigr),
\end{equation}
where $\sigma>0$ is a small constant. Noting that the function $\frac{ a B_1  }{ r^\alpha}  - B_2 U(|x_2-x_1|)$ has a maximum point in $S_k$. Therefore,
we can prove that \eqref{1.4} has a solution $u_k$
of the form
\[
u_k = W_{r_k}+\omega_k,
\]
where  $\omega_k\in \tilde E_k \cap  H^1(\R^N)$,   $r_k \in S_k$ and as $k\to +\infty$,

\begin{equation}\label{1-28-2}
\|\omega_k\|_* \le \frac{C }{r_k^\alpha}+C e^{- \min[\frac p2-\tau, 1]|x_2-x_1|}.
\end{equation}

\begin{remark}
In \cite{WY}, the reduction argument is carried out in $H^1(\mathbb R^N)\cap H_s$. Thus, we only obtain the estimate for
the error term $\omega_k$ in the norm of $H^1(\mathbb R^N)$.  Estimate \eqref{1-28-2} is a point-wise estimate, which is needed
when using the local Pohozaev identities.

\end{remark}

Next, assuming that $V$ satisfies \eqref{V2}, we derive a relation for $r_k$, which is needed in the proof of the non-degeneracy result.
 From

 \[
\int_{\mathbb R^N}  \bigl(  \nabla ( W_{r_k}+\omega_k )\nabla \frac{\partial U_{x_1}}{\partial y_1} + V(y) ( W_{r_k}+\omega_k )  \frac{\partial U_{x_1}}{\partial y_1}- ( W_{r_k}+\omega_k )^p
\frac{\partial U_{x_1}}{\partial y_1}\bigr)=0,
\]
we obtain

\[
\int_{\mathbb R^N}  \bigl( ( V(y) -1)( W_{r_k}+\omega_k )  \frac{\partial U_{x_1}}{\partial y_1}- \int_{\mathbb R^N}\Bigl( ( W_{r_k}+\omega_k )^p- \sum_{j=1}^k U^p_{x_j}\Bigr)
\frac{\partial U_{x_1}}{\partial y_1}\bigr)=0.
\]

It is easy to check that

\[
\begin{split}
&\int_{\mathbb R^N}  \bigl( ( V(y) -1)( W_{r_k}+\omega_k )  \frac{\partial U_{x_1}}{\partial y_1}\\
=&\int_{\mathbb R^N}   ( V(y) -1)U_{x_1}  \frac{\partial U_{x_1}}{\partial y_1}+ O\Bigl( \frac{1 }{r_k^{2\alpha}}+ e^{- \min[ p-2\tau, 2]|x_2-x_1|}\Bigr)\\
=&-\frac12 \int_{\mathbb R^N}   \frac{\partial  V(y) }{\partial y_1}U_{x_1}^2+ O\Bigl( \frac{1 }{r_k^{2\alpha}}+ e^{- \min[ p-2\tau, 2]|x_2-x_1|}\Bigr)\\
=&\frac{a (\alpha+1) x_1}{2r_k^{\alpha+2}}\int_{\mathbb R^N} U^2 + O\Bigl(\frac{1 }{r_k^{\alpha+2}}+ \frac{1 }{r_k^{2\alpha}}+ e^{- \min[ p-2\tau, 2]|x_2-x_1|}\Bigr).
\end{split}
\]
Moreover

\[
\begin{split}
&\int_{\mathbb R^N}\Bigl( ( W_{r_k}+\omega_k )^p- \sum_{j=1}^k U^p_{x_j}\Bigr)
\frac{\partial U_{x_1}}{\partial y_1}\bigr)\\
=&-p\int_{\mathbb R^N} U_{x_1}^{p-1} \sum_{j=2}^k U_{x_j}
\frac{\partial U_{x_1}}{\partial y_1}+O\Bigl(  \frac{1 }{r_k^{2\alpha}}+  e^{- \min(p-\tau, 2)|x_2-x_1|}\Bigr)\\
=&\int_{\mathbb R^N} U_{x_1}^{p} \sum_{j=2}^k
\frac{\partial U_{x_j}}{\partial y_1}+O\Bigl( \frac{1 }{r_k^{2\alpha}}+ e^{- \min(p-\tau, 2)|x_2-x_1|}\Bigr)\\
=& \bigl(B +o(1)\bigr)U(|x_2-x_1|)\Bigl( \frac{(x_1-x_2)_1}{|x_1-x_2|}+\frac{(x_1-x_k)_1}{|x_1-x_k|}\Bigr)+O\Bigl( \frac{1 }{r_k^{2\alpha}}+ e^{- \min(p-\tau, 2)|x_2-x_1|}\Bigr).
\end{split}
\]
Therefore, we have

\begin{equation}\label{1-18-4}
\begin{split}
&\frac{a (\alpha+1) x_1}{2r_k^{\alpha+2}}\int_{\mathbb R^N} U^2
\\
=&  \bigl(B +o(1)\bigr) U(|x_2-x_1|)\Bigl( \frac{(x_1-x_2)_1}{|x_1-x_2|}+\frac{(x_1-x_k)_1}{|x_1-x_k|}\Bigr)\\
&+O\Bigl( \frac{1 }{r_k^{2\alpha}}+\frac1{r_k^{\alpha+2}}+ e^{- \min(p-2\tau, 2)|x_2-x_1|}\Bigr).
\end{split}
\end{equation}

Note that

\[
(x_1-x_2)_1=(x_1-x_k)_1= r(1-\cos\frac{2\pi}k)= r\bigl( \frac{2\pi^2}{k^2}+ O\bigl(\frac1{k^4}\bigr)\bigr).
\]
 Thus,  \eqref{1-18-4} can be written as

 \begin{equation}\label{2-18-4}
\begin{split}
&\frac{a (\alpha+1) }{2r_k^{\alpha+1}}\int_{\mathbb R^N} U^2
\\
=& \bigl(B +o(1)\bigr) U(|x_2-x_1|)\frac1{k}+O\Bigl( \frac{1 }{r_k^{2\alpha}}+\frac1{r_k^{\alpha+2}}+ e^{- \min(p-2\tau, 2)|x_2-x_1|}\Bigr).
\end{split}
\end{equation}

Noting that   $e^{-|x_1-x_2|}\sim \frac{1 }{r_k^{\alpha}}$  and $\alpha>\frac4{p-1}$, it holds

\[
 e^{- \min(p-2\tau, 2)|x_2-x_1|}\le   \frac{C }{r_k^{\alpha\min(p-2\tau, 2)}}= o\bigl(\frac{1 }{r_k^{\alpha+1}}\bigr)
 \]

\section{The non-degeneracy of the solutions}

Let $u$ solve

 \begin{equation}\label{1-26-11}
  -\Delta u + V(|y|)u=u^p,
 \end{equation}
 and let $\xi$ solve

  \begin{equation}\label{2-26-11}
  -\Delta \xi +  V(|y|)\xi =p u^{p-1}\xi.
 \end{equation}

 \begin{lemma}\label{l0-14-12}
 We have

  \begin{equation}\label{1-20-12}
 \begin{split}
 & - \int_{\partial \Omega}\frac{\partial u}{\partial \nu} \frac{\partial \xi}{\partial y_i}- \int_{\partial \Omega}\frac{\partial \xi}{\partial \nu} \frac{\partial u}{\partial y_i}+
 \int_{\partial \Omega} \bigl\langle \nabla u, \nabla \xi\bigr\rangle \nu_i\\
 & +\int_{\partial \Omega}V(|y|) u \xi \nu_i  -\int_{\partial \Omega} u^p \xi \nu_i=\int_\Omega u\xi \frac{\partial V }{\partial y_i}
 \end{split}
 \end{equation}

 \end{lemma}

 \begin{proof}

 We have

 \begin{equation}\label{3-26-11}
 \begin{split}
 & \int_\Omega\Bigl( -\Delta u \frac{\partial \xi}{\partial y_i} + (-\Delta \xi) \frac{\partial u }{\partial y_i}\Bigr)+
 \int_\Omega V(|y|)\Bigl(  u \frac{\partial \xi}{\partial y_i} +  \xi \frac{\partial u }{\partial y_i}\Bigr)
 \\
 =&\int_\Omega \Bigl(  u^p  \frac{\partial \xi}{\partial y_i} +  p u^{p-1} \xi \frac{\partial u }{\partial y_i}\Bigr).
 \end{split}
 \end{equation}

 It is easy to check that

  \begin{equation}\label{4-26-11}
 \int_\Omega \Bigl(  u^p  \frac{\partial \xi}{\partial y_i} +  p u^{p-1} \xi \frac{\partial u }{\partial y_i}\Bigr)
 = \int_\Omega  \frac{\partial ( u^p \xi) }{\partial y_i}=\int_{\partial \Omega} u^p \xi \nu_i,
 \end{equation}
and

\begin{equation}\label{5-26-11}
\begin{split}
 &
 \int_\Omega V(|y|)\Bigl(  u \frac{\partial \xi}{\partial y_i} +  \xi \frac{\partial u }{\partial y_i}\Bigr)= \int_\Omega V(|y|) \frac{\partial ( u \xi) }{\partial y_i}\\
 =& - \int_\Omega u\xi \frac{\partial V }{\partial y_i}+ \int_{\partial \Omega}V(|y|) u \xi \nu_i.
 \end{split}
 \end{equation}
Moreover,

\begin{equation}\label{6-26-11}
 \begin{split}
 & \int_\Omega\Bigl( -\Delta u \frac{\partial \xi}{\partial y_i} + (-\Delta \xi) \frac{\partial u }{\partial y_i}\Bigr)\\
 =& - \int_{\partial \Omega}\frac{\partial u}{\partial \nu} \frac{\partial \xi}{\partial y_i} + \int_\Omega \frac{\partial u}{\partial y_j} \frac{\partial^2 \xi}{\partial y_i\partial
 y_j}
 - \int_{\partial \Omega}\frac{\partial \xi}{\partial \nu} \frac{\partial u}{\partial y_i} + \int_\Omega \frac{\partial \xi}{\partial y_j} \frac{\partial^2 u}{\partial y_i\partial
 y_j}\\
 =& - \int_{\partial \Omega}\frac{\partial u}{\partial \nu} \frac{\partial \xi}{\partial y_i}- \int_{\partial \Omega}\frac{\partial \xi}{\partial \nu} \frac{\partial u}{\partial y_i}
 +\int_\Omega \frac{\partial }{\partial y_i} \Bigl(  \frac{\partial u}{\partial y_j} \frac{\partial \xi}{\partial y_j} \Bigr)\\
 =& - \int_{\partial \Omega}\frac{\partial u}{\partial \nu} \frac{\partial \xi}{\partial y_i}- \int_{\partial \Omega}\frac{\partial \xi}{\partial \nu} \frac{\partial u}{\partial y_i}+
 \int_{\partial \Omega} \bigl\langle \nabla u, \nabla \xi\bigr\rangle \nu_i.
 \end{split}
 \end{equation}
 So we have proved \eqref{1-20-12}.

\end{proof}

Let $u_k$ be a solution of

 \begin{equation}\label{1-27-11}
  -\Delta u + V(|y|)u=u^p,
 \end{equation}
constructed in Theorem~A.

We now  prove Theorem~\ref{th11}, arguing  by contradiction.   Suppose that there are $k_m\to +\infty$,  satisfying $\xi_m\in H_s$,
 $\|\xi_m\|_*=1$, and
\begin{equation}\label{10-13-12}
 L_{k_m}\xi_m =0.
 \end{equation}

Let
\begin{equation}\label{1-10-4}
\tilde \xi_{m} (y)= \xi_m
(y+ x_{k_m, 1}).
\end{equation}

\begin{lemma}\label{l1-10-4}

 It holds

\begin{equation}\label{10-9-4}
\tilde \xi_{m}  \to  b\frac{\partial U}{\partial y_1},
\end{equation}
uniformly in $C^1(B_R(0))$ for any $R>0$, where $b$ is  some constant.

\end{lemma}

\begin{proof}

  In view of
$|\tilde \xi_{m}|\le C$, we may assume that $\tilde \xi_{m}\to \xi$ in
$ C_{loc}(\mathbb R^N)$.   Then $\xi$ satisfies

\begin{equation}\label{3-8-4}
-\Delta \xi+ \xi  = p U^{p-1}\xi, \quad \text{in}\; \mathbb R^N.
\end{equation}
This gives

\begin{equation}\label{4-8-4}
\xi =\sum_{i=1}^{N} b_{i}\frac{\partial U}{\partial y_i}.
\end{equation}

Since $\xi$ is even in $y_j$ for $j=2,\cdots,N$, it holds $b_2=\cdots=b_N=0$.

\end{proof}

We decompose

\[
 \xi_{m} (y)= b_{m}  \sum_{j=1}^k \frac{\partial U_{x_j}}{\partial r}+ \xi_m^*,
\]
where
\[
\xi_m^*\in E_k=\bigl\{ v:  v\in  H_s\cap C(\mathbb R^N),\; \sum_{j=1}^k \int_{\R^N} U_{x_j}^{p-1} Z_j v=0, \;j=1,\cdots,k\bigr\}.
\]
 By Lemma~\ref{l1-10-4},  $b_m$ is bounded.

\begin{lemma}\label{l10-28-2}

 It holds

\[
\|\xi_m^*\|_*\le \frac{C}{ |x_1|^\alpha}+  C e^{- \min[\frac p2-\tau, 1]|x_2-x_1|}.
\]

\end{lemma}

\begin{proof}

It is easy to see that

\[
\begin{split}
 L_{k_m} \xi^*_m = & b_m  \bigl(  V(|y|) -1\bigr) \sum_{j=1}^k \frac{\partial U_{x_j}}{\partial r}\\
&-
 b_m  p  \sum_{j=1}^k \bigl(  u_{k_m}^{p-1}
  -U_{x_j}^{p-1}\bigr)  \frac{\partial U_{x_j}}{\partial r}.
\end{split}
\]

Similar to the proof of lemma~\ref{100l1-25-3}, using \eqref{1-28-2},
we can  prove

\[
\|\bigl(  V(|y|) -1\bigr) \sum_{j=1}^k \frac{\partial U_{x_j}}{\partial r} \|_*\le \frac{C}{ |x_1|^\alpha},
\]
and

\[
\| \sum_{j=1}^k \bigl(  u_{k_n}^{p-1}(y)
  -U_{x_j}^{p-1}\bigr)  \frac{\partial U_{x_j}}{\partial r}
\|_*\le C e^{- \min[\frac p2-\tau, 1]|x_2-x_1|}.
\]

Moreover, from  $\xi_m^*\in E_k$, similar to Lemma~\ref{l21}, we can prove that there exists $\rho>0$, such that

\[
\| \tilde L_k \xi^*_n\|_*\ge  \rho \|  \xi^*_n\|_*.
\]
Thus, the result follows.

\end{proof}

\begin{lemma}\label{l1-14-12}

 It holds

\begin{equation}\label{1-1-4-12}
\tilde \xi_{m}\to 0,
\end{equation}
uniformly in $C^1(B_R(0))$ for any $R>0$.

\end{lemma}

\begin{proof}

We apply the identities in Lemma~\ref{l0-14-12}  in the domain  $\Omega=B_{\frac12 |x_2-x_1|}(x_1)$.  For simplicity, we drop the subscript $m$.

We define the following bilinear form

  \begin{equation}\label{1-20-3}
 \begin{split}
L(u, \xi, \Omega)= & - \int_{\partial \Omega}\frac{\partial u}{\partial \nu} \frac{\partial \xi}{\partial y_1}- \int_{\partial \Omega}\frac{\partial \xi}{\partial \nu} \frac{\partial u}{\partial y_1}+
 \int_{\partial \Omega} \bigl\langle \nabla u, \nabla \xi\bigr\rangle \nu_1\\
 & +\int_{\partial \Omega} u \xi \nu_1,
 \end{split}
 \end{equation}
 where $\nu$ is the outward unit normal of $\partial\Omega$.
Then it follows Lemma~\ref{l0-14-12} that

  \begin{equation}\label{2-20-3}
L(u, \xi,  \Omega)  +\int_{\partial \Omega} \bigl( V(|y|)-1\bigr)u \xi \nu_1-\int_{\partial \Omega} u^p \xi \nu_1=\int_\Omega u\xi \frac{\partial V }{\partial y_1}.
 \end{equation}

Let recall that

 \begin{equation}\label{3-20-3}
 \begin{split}
L(u, \xi,  \Omega) = &\int_\Omega \bigl( -\Delta u +   u - u^p\bigr) \frac{\partial \xi}{\partial y_1}\\
& +
 \int_\Omega \bigl( -\Delta \xi  +  \xi -p u^{p-1} \xi \bigr)\frac{\partial u }{\partial y_1}+\int_{\partial \Omega} u^p \xi \nu_1.
 \end{split}
 \end{equation}

Note that

\[
u_{k_m}=\sum_{j=1}^k  U_{x_j}(y)+\omega_{k_m},
\]
and

\[
 \xi_{m} (y)= b_m\sum_{j=1}^k  \frac{\partial U_{x_j}}{\partial r}+ \xi_m^*,
\]

We have

 \begin{equation}\label{3-27-4}
 \begin{split}
 &\frac{\partial U_{x_j}}{\partial r}= \frac{\partial U\bigl(y- ( r\cos\frac{2(j-1)\pi}k, r\sin\frac{2(j-1)\pi}k, 0)\bigr)}{\partial r}\\
 =&  -U'(|y- x_j|) \bigl\langle \frac{y-x_j}{|y-x_j|}, ( \cos\frac{2(j-1)\pi}k, \sin\frac{2(j-1)\pi}k, 0)\bigr\rangle,
 \end{split}
 \end{equation}
and

 \[
 -\Delta \frac{\partial U_{x_j}}{\partial r}  +  \frac{\partial U_{x_j}}{\partial r} -p U_{x_j}^{p-1}\frac{\partial U_{x_j}}{\partial r}=0.
 \]

We now calculate

 \begin{equation}\label{4-20-3}
 \begin{split}
& L(U_{x_j} ,  \frac{\partial U_{x_j}}{\partial r},  \Omega)
\\
= &\int_\Omega \bigl( -\Delta U_{x_j} +   U_{x_j} - U_{x_j}^p\bigr) \frac{\partial \frac{\partial U_{x_j}}{\partial r}}{\partial y_1}\\
& +
 \int_\Omega \bigl( -\Delta \frac{\partial U_{x_j}}{\partial r}  +  \frac{\partial U_{x_j}}{\partial r} -p U_{x_j}^{p-1}\frac{\partial U_{x_j}}{\partial r}\bigr)\frac{\partial U_{x_1} }{\partial y_1}+\int_{\partial \Omega} U_{x_j}^p \frac{\partial U_{x_j}}{\partial r} \nu_1\\
 =&\int_{\partial \Omega} U_{x_j}^p \frac{\partial U_{x_j}}{\partial r} \nu_1= O\bigl( e^{-(\frac {p+1}2 -\sigma)|x_2-x_1|}\bigr),
 \end{split}
 \end{equation}
where $\sigma>0$ is any fixed small constant.

On the other hand, for $i\ne j$, we have

  \begin{equation}\label{5-20-3}
 \begin{split}
& L(U_{x_i} ,  \frac{\partial U_{x_j}}{\partial r},  \Omega)
\\
= &
 \int_\Omega \bigl( -\Delta \frac{\partial U_{x_j}}{\partial r}  +  \frac{\partial U_{x_j}}{\partial r} -p U_{x_i}^{p-1}\frac{\partial U_{x_j}}{\partial r}\bigr)\frac{\partial U_{x_i} }{\partial y_1}
 +   O\bigl( e^{-(\frac {p+1}2 -\sigma)|x_2-x_1|}\bigr)\\
 =&p \int_\Omega U_{x_j}^{p-1}\frac{\partial U_{x_j}}{\partial r}\frac{\partial U_{x_i} }{\partial y_1}  -p \int_\Omega U_{x_i}^{p-1}\frac{\partial U_{x_j}}{\partial r}\frac{\partial U_{x_i} }{\partial y_1}+O\bigl( e^{-(\frac {p+1}2 -\sigma) |x_2-x_1|}\bigr).
 \end{split}
 \end{equation}

 It is easy to check  that for $j\ge 2$,  $i\ne j$,

 \begin{equation}\label{1-27-4}
  \int_\Omega U_{x_j}^{p-1}\frac{\partial U_{x_j}}{\partial r}\frac{\partial U_{x_i} }{\partial y_1}= O\Bigl( e^{-(1+ \frac{p-1}2 -\sigma )|x_1-x_2|}\Bigr),
 \end{equation}
and

 \begin{equation}\label{2-27-4}
  \int_\Omega U_{x_j}^{p-1}\frac{\partial U_{x_i}}{\partial r}\frac{\partial U_{x_j} }{\partial y_1}= O\Bigl( e^{-(1+ \frac{p-1}2 -\sigma )|x_1-x_2|}\Bigr),
 \end{equation}
 where $\sigma>0$ is any fixed small constant.

 Combining \eqref{4-20-3}--\eqref{2-27-4}, we obtain

  \begin{equation}\label{10-27-4}
 \begin{split}
& L(\sum_{j=1}^k U_{x_j} ,  \sum_{i=1}^k \frac{\partial U_{x_i}}{\partial r},  \Omega)
\\
 =&p \int_\Omega U_{x_1}^{p-1}\frac{\partial U_{x_1}}{\partial r}\Bigl( \frac{\partial U_{x_2} }{\partial y_1} +\frac{\partial U_{x_k} }{\partial y_1}\Bigr)
  -p \int_\Omega U_{x_1}^{p-1}\Bigl(\frac{\partial U_{x_2}}{\partial r}+  \frac{\partial U_{x_k}}{\partial r} \Bigr)
  \frac{\partial U_{x_1} }{\partial y_1}\\
  &+O\bigl( e^{-(\frac {p+1}2 -\sigma)|x_2-x_1|}\bigr)\\
 =& -p \int_\Omega U_{x_1}^{p-1}\frac{\partial U_{x_1}}{\partial y_1}\Bigl(\frac{\partial U_{x_2} }{\partial y_1}+\frac{\partial U_{x_2}}{\partial r}
 +\frac{\partial U_{x_k} }{\partial y_1}+\frac{\partial U_{x_k}}{\partial r}\Bigr)
 +O\bigl( e^{-(\frac {p+1}2 -\sigma)|x_2-x_1|}\bigr)\\
 =&  \int_{\mathbb R^N} U_{x_1}^{p}\frac{\partial }{\partial y_1}\Bigl(\frac{\partial U_{x_2} }{\partial y_1}+\frac{\partial U_{x_2}}{\partial r}
 +\frac{\partial U_{x_k} }{\partial y_1}+\frac{\partial U_{x_k}}{\partial r}\Bigr)
 +O\bigl( e^{-(\frac {p+1}2 -\sigma)|x_2-x_1|}\bigr)\\
 = &\bigl( B+o(1)\bigr) \frac{\partial }{\partial y_1}\Bigl(\frac{\partial U_{x_2} }{\partial y_1}+\frac{\partial U_{x_2}}{\partial r}
 +\frac{\partial U_{x_k} }{\partial y_1}+\frac{\partial U_{x_k}}{\partial r}\Bigr)\Bigr|_{ y= x_1}
 +O\bigl( e^{-(\frac {p+1}2 -\sigma)|x_2-x_1|}\bigr),
 \end{split}
 \end{equation}
where $B>0$ is a constant.

  By \eqref{3-27-4}, we have

  \[
 \begin{split}
& \frac{\partial U_{x_2} }{\partial y_1}+\frac{\partial U_{x_2}}{\partial r}
 \\
 = &  U'(|y- x_2|) \bigl\langle \frac{y-x_2}{|y-x_2|}, (1, 0,0)-( \cos\frac{2\pi}k, \sin\frac{2\pi}k, 0)\bigr\rangle.
 \end{split}
 \]
Therefore,

 \[
 \begin{split}
& \frac{\partial }{\partial y_1}\Bigl( \frac{\partial U_{x_2} }{\partial y_1}+\frac{\partial U_{x_2}}{\partial r}\Bigr)
 \\
 = &  U''(|y- x_2|)\frac{(y-x_2)_1}{|y-x_2|} \bigl\langle \frac{y-x_2}{|y-x_2|}, (1, 0,0)-( \cos\frac{2\pi}k, \sin\frac{2\pi}k, 0)\bigr\rangle\\
 &+U'(|y- x_2|) \bigl\langle \frac{(1,0,0)}{|y-x_2|}- (y-x_2)\frac{(y-x_2)_1}{|y-x_2|^3}, (1, 0,0)-( \cos\frac{2\pi}k, \sin\frac{2\pi}k, 0)\bigr\rangle.
 \end{split}
 \]
We have

\[
 (1, 0,0)-( \cos\frac{2\pi}k, \sin\frac{2\pi}k, 0)= \bigl( \frac{2\pi^2}{k^2}+O(\frac1{k^4}), \frac{2\pi}k+O(\frac1{k^2}), 0\bigr).
 \]

Note that

\[
|x_2-x_1|=2r_k\sin\frac{\pi}k= \frac {2\pi r_k}k+O\bigl( \frac{ r_k}{k^3}\bigr),
\]

\[
(x_1-x_2)_1= r_k \bigl(1- \cos\frac{2\pi}k\bigr)= \frac{2\pi^2 r_k}{k^2}+O\bigl( \frac{ r_k}{k^4}\bigr),
\]
and

\[
(x_1-x_2)_2= -r_k \sin \frac{2\pi}k =- \frac {2\pi r_k}{k}+O\bigl( \frac{r_k}{k^3}\bigr).
\]
So we obtain

 \[
 \begin{split}
& \frac{\partial }{\partial y_1}\Bigl( \frac{\partial U_{x_2} }{\partial y_1}+\frac{\partial U_{x_2}}{\partial r}\Bigr)\Bigr|_{y=x_1}
 \\
 = &  U''(|x_1- x_2|)\frac{(x_1-x_2)_1}{|x_1-x_2|}  \frac{(x_1-x_2)_2}{|x_1-x_2|}\bigl(- \sin\frac{2\pi}k\bigr)\\
 & +U'(|x_1- x_2|) \Bigl( \frac{1}{|x_1-x_2|}  \bigl( 1-\cos\frac{2\pi}k\bigr)  + \frac{(x_1-x_2)_2(x_1-x_2)_1}{|x_1-x_2|^3} \sin\frac{2\pi}k\Bigr)+O\bigl(\frac1{k^3}\bigr)\\
 =&  U(|x_1- x_2|)\frac{(x_1-x_2)_1}{|x_1-x_2|}\Bigl(\sin\frac{2\pi}k- \frac1{r_k}+ \frac{1}{|x_1-x_2|} \sin\frac{2\pi}k
 \Bigr)\Bigl( 1+o(1)\Bigr)\\
 =& \frac{ 2\pi^2 U(|x_1- x_2|)}{k^2}\Bigl( 1+o(1)\Bigr).
 \end{split}
 \]
Similarly

 \[
 \begin{split}
& \frac{\partial }{\partial y_1}\Bigl( \frac{\partial U_{x_k} }{\partial y_1}+\frac{\partial U_{x_k}}{\partial r}\Bigr)\Bigr|_{y=x_1}
 \\
 = &  U''(|x_1- x_k|)\frac{(x_1-x_k)_1}{|x_1-x_k|}  \frac{(x_1-x_k)_2}{|x_1-x_k|}\bigl(- \sin\frac{2(k-1)\pi}k\bigr)\\
 & +U'(|x_1- x_k|) \Bigl( \frac{1}{|x_1-x_k|}  \bigl( 1-\cos\frac{2(k-1)\pi}k\bigr)  + \frac{(x_1-x_k)_2(x_1-x_k)_1}{|x_1-x_k|^3} \sin\frac{2(k-1)\pi}k\Bigr)\\
 &+O\bigl(\frac1{k^3}\bigr)\\
 =& \frac{ 2\pi^2 U(|x_1- x_k|)}{k^2}\Bigl( 1+o(1)\Bigr).
 \end{split}
 \]
So we have proved

  \begin{equation}\label{60-27-4}
 \begin{split}
& L(\sum_{j=1}^k U_{x_j} ,  \sum_{i=1}^k \frac{\partial U_{x_i}}{\partial r},  \Omega)
\\
 = &\bigl( B'+o(1)\bigr)  \frac{  U(|x_1- x_k|)}{k^2}
 +O\bigl( e^{-(\frac {p+1}2 -\sigma)|x_2-x_1|}\bigr).
 \end{split}
 \end{equation}

 Using \eqref{1-28-2}  and Lemma~\ref{l10-28-2}, together with $L^p$--estimates for the elliptic equations,
 we can prove

\[
|\nabla \omega_{k_m}(y)|,\;\; |\nabla \xi_m^*(y)|\le \frac{C}{ |x_1|^\alpha}+  C e^{- \min[\frac p2-\tau, 1]|x_2-x_1|},\quad \forall y\in\partial\Omega,
\]
from which, we see

\begin{equation}\label{70-27-4}
L(\omega_{k_m},  \sum_{i=1}^k \frac{\partial U_{x_i}}{\partial r},  \Omega)= O\bigl( \frac{1}{ |x_1|^\alpha}+  C e^{- \min[\frac p2-\tau, 1]|x_2-x_1|}  \bigr)e^{-\frac{1-\sigma}2
|x_1-x_2|}
 \end{equation}
 and

\begin{equation}\label{71-27-4}
L(\sum_{j=1}^k U_{x_j} ,  \xi^*_m,  \Omega)= O\bigl( \frac{1}{ |x_1|^\alpha}+  C e^{- \min[\frac p2-\tau, 1]|x_2-x_1|}  \bigr)e^{-\frac{1-\sigma}2
|x_1-x_2|}
 \end{equation}

It is easy to check that

  \begin{equation}\label{1-28-4}
 \int_{\partial \Omega} \bigl( V(|y|)-1\bigr)u \xi \nu_1=O\bigl(\frac1{r_k^\alpha} e^{-(1-\sigma)|x_1-x_2|}\bigr),
 \end{equation}
 and

  \begin{equation}\label{2-28-4}
 \int_{\partial \Omega} u^p \xi \nu_1=O\bigl(e^{-(\frac{p+1}2-\sigma)|x_1-x_2|}\bigr).
 \end{equation}

 Combining \eqref{2-20-3}, and \eqref{60-27-4}--\eqref{2-28-4}, we obtain

  \begin{equation}\label{3-28-4}
 \begin{split}
b_m\bigl( B'+o(1)\bigr)  \frac{  U(|x_1- x_k|)}{k^2}
 +O\bigl( \frac1{r_k^\alpha} e^{-(1-\sigma)|x_1-x_2|}+e^{-(\frac {p+1}2 -\sigma)|x_2-x_1|}\bigr)=\int_\Omega u\xi \frac{\partial V }{\partial y_1}.
 \end{split}
 \end{equation}

 Using \eqref{1-28-2}  and  Lemma~\ref{l10-28-2}, we find

  \begin{equation}\label{1000-21-12}
   \begin{split}
 &\int_{ \Omega} u\xi V'(|y|) |y|^{-1} y_1
 \\
 =&b_m\int_{ \mathbb R^N} U_{x_1}   \frac{\partial U_{x_1}}{\partial y_1} V'(|y|) |y|^{-1} y_1  +\frac1{|x_1|^{\alpha+1}} O\Bigl(\frac{1}{ |x_1|^\alpha}+   e^{- \min[(\frac p2-\sigma), 1]|x_2-x_1|}\Bigr).
 \end{split}
 \end{equation}

 For any $h>1$, it holds

 \[
 \frac{ y_1+ x_{1, 1}} {|y+x_1|^h} = \frac{  x_{1, 1}} {|x_1|^h}- (h-1) \frac{ y_1} {|x_1|^h}+O\bigl(\frac1{|x_1|^{h+1}}\bigr).
 \]
Thus

\[
\int_{\mathbb R^N}U  \frac{\partial U}{\partial y_1}  \frac{ y_1+ x_{1, 1}} {|y+x_1|^h}=- (h-1)\int_{\mathbb R^N}U  \frac{\partial U}{\partial y_1} \frac{ y_1} {|x_1|^h}
+O\bigl(\frac1{|x_1|^{h+1}}\bigr).
\]
 Therefore, by \eqref{V2},

 \begin{equation}\label{11-21-12}
 \begin{split}
  &\int_{ \mathbb R^N} U_{x_1}   \frac{\partial U_{x_1}}{\partial y_1} V'(|y|) |y|^{-1} y_1\\
   = & \int_{ \mathbb R^N+ x_1 } U  \frac{\partial U}{\partial y_1}V'(|y+ x_1|) |y+ x_1|^{-1} (y_1+  x_{1,1})\\
    = & \int_{ \mathbb R^N } U  \frac{\partial U}{\partial y_1} \Bigl( -\frac{ \alpha a_1 }{ |y+ x_1|^{\alpha+1}} -\frac{ (\alpha+1) a_2 } {|y+ x_1|^{\alpha+2}} +
    O\bigl( \frac{ 1 }{ |y+ x_1|^{m+3}}\bigr)\Bigr)\frac{ y_1+  x_{1,1}}{|y+ x_1| }
  \\
  =&\frac {\alpha(\alpha+1)a_1}{|r_{k_n}|^{\alpha+2}} \Bigl( \int_{\mathbb R^N}  U U'(|y|)  y_1^2  +o(1)\Bigr).
  \end{split}
 \end{equation}

 Combining \eqref{3-28-4} and \eqref{11-21-12}, we obtain

 \begin{equation}\label{1000-21-12}
 \begin{split}
  & b_m \frac {\alpha(\alpha+1)a_1}{|r_{k_n}|^{\alpha+2}} \Bigl( \int_{\mathbb R^N}  U U'(|y|)  y_1^2  +o(1)\Bigr)
  +\frac1{|x_1|^{\alpha+1}} O\Bigl(\frac{1}{ |x_1|^\alpha}+   e^{- \min[(\frac p2-\sigma), 1]|x_2-x_1|}\Bigr)
  \\
  =& b_m \bigl( B'+o(1)\bigr)  \frac{  U(|x_1- x_k|)}{k^2}
 +O\bigl( \frac1{r_k^\alpha} e^{-(1-\sigma)|x_1-x_2|}+e^{-(\frac {p+1}2 -\sigma)|x_2-x_1|}\bigr).
  \end{split}
 \end{equation}

 Since $\alpha>\max\bigl(\frac 4{p-1}, 2\bigr)$, we see that

 \[
 \frac1{r_k^\alpha} e^{-(1-\sigma)|x_1-x_2|}+e^{-(\frac {p+1}2 -\sigma)|x_2-x_1|}=o\Bigl(  \frac{  U(|x_1- x_k|)}{k^2}  \Bigr).
 \]

 By \eqref{2-18-4},

  \begin{equation}\label{10-28-4}
\frac{1 }{r_k^{\alpha+1}}
\sim U(|x_2-x_1|)\frac1{k}.
\end{equation}
Hence

\[
\frac{  U(|x_1- x_k|)}{k^2}\sim \frac1k \frac{1 }{r_k^{\alpha+1}}=\frac{r_k}k \frac{1 }{r_k^{\alpha+2}}>>\frac{1 }{r_k^{\alpha+2}},
\]
since $r_k\sim k\ln k$.  Thus, \eqref{1000-21-12} gives $b_m\to 0$.

\end{proof}

\begin{proof} [Proof of Theorem~\ref{th11}]

Let  $G(y, x)$ be the Green's function of $-\Delta u + V(|y|) u$  in $\mathbb R^N$.  Then

 \[
 0< G(z, y) \le  C e^{-|z-y|}
 \]
 as $|z-y|\to +\infty$.  We have

\begin{equation}\label{2-28-11}
\xi_k (y)= p \int_{\mathbb R^N}  G(z, y) u_k^{p-1}(z) \xi_k(z)\,dz.
\end{equation}

Now we estimate

\[
\begin{split}
&|\int_{\mathbb R^N}  G(z, y) u_k^{p-1}(z) \xi_k(z)\,dz|\le  C\|\xi _k\|_{*} \int_{\mathbb R^N} G(z, y) u_k^{p-1}(z)\sum_{j=1}^k e^{-\tau|z-x_j|}\,dz
\\
\le & C\|\xi _k\|_{*} \int_{\mathbb R^N} G(z, y)\bigl( W_r^{p-1}(z)+ |\omega|^{p-1}\bigr)\sum_{j=1}^k e^{-\tau|z-x_j|}\,dz\\
\le & C\|\xi _k\|_{*} \int_{\mathbb R^N} G(z, y) \Bigl(\sum_{j=1}^k e^{- (p-1-\sigma+\tau)|z-x_j|}+\|\omega\|_*^{p-1} \sum_{j=1}^k e^{-(p\tau-\sigma) |z-x_j|}\Bigr)
\\
\le & C\|\xi_k\|_{*} \sum_{j=1}^k e^{-(p\tau-\sigma)|y-x_j|},
\end{split}
\]
where $\sigma\in (0, \tau)$ is any fixed small constant. So we obtain

\[
\frac{ |\xi_k(y)|}{ \sum_{j=1}^k e^{-\tau |y-x_j|}}\le C\|\xi_k\|_{*}\frac{ \sum_{j=1}^k e^{-(p\tau-\sigma) |y-x_j|}}{ \sum_{j=1}^k e^{-\tau |y-x_j|}}.
\]

Since $\xi_k\to 0$ in $B_R(x_{j, k})$  and $\|\xi_k\|_*=1$, we know that $\frac{ |\xi_k(y)|}{ \sum_{j=1}^k e^{-\tau |y-x_j|}}$
attains its maximum in $\mathbb R^N \setminus \cup_{j=1}^k B_R(x_{j, k})$. Thus

\[
\|\xi_k\|_* \le o(1)\|\xi_k\|_* .
\]
So $\|\xi_k\|_*\to 0$ as $k\to +\infty$. This is a contradiction to $\|\xi_k\|_*=1$.

\end{proof}

\section{Existence of new solutions}
\setcounter{equation}{0}

Let  $u_k$ be the  solutions constructed in section~2, where $k>0$ is a large even integer. Since $k$ is even, $u_k$ is even in each
$y_j$, $j=1,\cdots, N$.   Moreover,  $u_k$  is radial in $y''=(y_3, \cdots, y_N)$.

Let $n\ge k$ be a large even integer.
Set
\[
p_j=\Bigl(0, 0, t \cos\frac{2(j-1)\pi}n, t\sin\frac{2(j-1)\pi}n,0\Bigr),\quad j=1,\cdots,n,
\]
where   $t \in \bigl[ t_0 n\ln n,\; t_1 n \ln n\bigr] $.

  Define
\[
\begin{split}
X_s=\Bigl\{ u: & u\in H_s, u\;\text{is even in} \;y_h, h=1,\cdots,N,\\
& u(y_1, y_2, t\cos\theta , t\sin\theta, y^*)=
u(y_1,y_2, t\cos(\theta+\frac{2\pi j}n) , t\sin(\theta+\frac{2\pi j}n), y^*)
\Bigr\}.
\end{split}
\]
Here $y^*=(y_5,\cdots,y_N)$.

Note that   both $u_k $  and $\sum_{j=1}^n U_{p_j}$ belong to $ X_s$, while  $u_k$  and $\sum_{j=1}^n U_{p_j, \lambda}$ are separated  from each other. We aim to construct a solution for \eqref{1.4} of the form
\[
u= u_k +\sum_{j=1}^n U_{p_j,\lambda} +\xi,
\]
where $\xi\in X_s$ is a small perturbed term.  Since the procedure to construct such solutions are similar to that in \cite{WY},
we just sketch it.

We define the linear operator
\begin{equation}\label{10-29-3}
 Q_n \xi= -\Delta \xi +V(|y|) \xi-p\Bigl(  u_k+\sum_{j=1}^n U_{p_j}\Bigr)^{p-1}\xi,\quad \xi\in X_s.
 \end{equation}
We can regard $Q_n \xi$ as a function in $X_s$, such that

\begin{equation}\label{10-29-3}
\bigl\langle  Q_n \xi, \phi\bigr\rangle = \int_{\mathbb R^N} \Bigl( \nabla \xi\nabla \phi
+V(|y|) \xi\phi -p\bigl(  u_k+\sum_{j=1}^n U_{p_j}\bigr)^{p-1}\xi\phi\Bigr),\quad \xi, \, \phi\in X_s.
 \end{equation}

Let
 \[
\tilde Z_{j}=\frac{\partial U_{p_j }}{\partial t},\;\;j=1,\cdots,k.
  \]

 Let  $h_n\in X_s$. Consider
\begin{equation}\label{20-2-4}
\begin{cases}
Q_n \xi_n = h_n +  c_{n} \sum\limits_{j=1}^n  \tilde   Z_{j}, \\
\xi_n \in X_s,\\
\displaystyle \sum_{j=1}^n \int_{\mathbb R^N} U_{p_j}^{p-1} \tilde Z_{j} \xi_n=0,
\end{cases}
\end{equation}
for some constants $c_{n}$, depending on $\xi_n$.

Let
\[
D_j=\Bigl\{ y=(y',y_3, y_4, y^*)\in\R^2\times \mathbb R^2\times \R^{N-4}:
 \Bigl\langle \frac {(0, 0, y_3, y_4, 0,\cdots,0)}{|(y_3, y_4)|}, \frac{p_{ j}}{|p_{ j}|}\Bigr\rangle\ge \cos \frac{\pi}{n}\Bigr\}.
\]

\begin{lemma}\label{l20-2-4}

Assume that $\xi_n$ solve \eqref{20-2-4}.  If $\|h_n\|_{H^1(\mathbb R^N)}\to 0$, then $\|\xi_n\|_{H^1(\mathbb R^N)}\to 0$.

\end{lemma}

\begin{proof}

 For simplicity, we will use  $\| u\|$ to denote $\|u \|_{H^1(\mathbb R^N)}$.

We argue by contradiction.  Suppose that there are $p_{n, j}$,   $h_n$  and $\xi_n$, satisfying \eqref{20-2-4},    $\|h_n\|\to 0$ and $\|\xi_n\|\ge c>0$.
We may assume  $\|\xi_n\|^2 =n$.  Then $\|h_n\|^2 = o(n)$.

First, we estimate  $c_n$.  We have
\[
c_n \sum_{j=1}^n \int_{\mathbb R^N}   U_{p_{n,j}}^{p-1}\tilde  Z_{j} \tilde  Z_{1}=  \bigl\langle Q_n \xi_n - h_n,   \tilde Z_{1}\bigr\rangle,
\]

from which we can proved  that $c_{n}=o(1)$.

 On the other hand,  from $\int_{\mathbb R^N} U_{p_{n,j},\lambda_n}^{2^*-2} Z_j \xi_n=0$, it is standard to prove that
\[
 \xi_n( y+ p_{n, j}) \to  0,\quad \text{in}\;  H^1_{loc}(\mathbb R^N).
 \]

 Moreover, since  $\frac{1}{\sqrt n} \xi_n$ is bounded in $H^1(\mathbb R^N)$, we can assume that
\[
 \frac{1}{\sqrt n} \xi_n\rightharpoonup \xi,\quad \text{weakly in }\; H^1(\mathbb R^N),
 \]
 and
\[
 \frac{1}{\sqrt n} \xi_n\to  \xi,\quad \text{strongly in }\; L^2_{loc}(\mathbb R^N).
 \]
Thus, $\xi$ satisfies
\[
-\Delta \xi+V(|y|) \xi -pu_k^{p-1}\xi =0, \quad \text{in}\; \mathbb R^N.
\]
By  Theorem~\ref{th11}, $\xi=0$.  Therefore,
\[
\int_{\mathbb R^N}  \bigl(  u_k+\sum_{j=1}^n U_{p_{n, j},\lambda_n}\bigr)^{p-1}\xi_n^2=o(n).
 \]

 We also have
 \[
 \langle  h_n,   \xi_n\bigr\rangle= o(n),
 \]
 and
\[
 \Bigl\langle  \sum_{j=1}^n    \tilde Z_{j}, \xi_n \Bigr\rangle  =O(n).
 \]
 So we obtain
\[
 \begin{split}
 & \int_{\mathbb R^N} |\nabla \xi_n|^2 \\
 =&O\Big(\int_{\mathbb R^N}  \Bigl(  u_k+\sum_{j=1}^n U_{p_{n, j},\lambda_n}\Bigr)^{p-1}\xi_n^2 \Big)+ \bigl\langle   h_n,  \xi_n\bigr\rangle +
 \int_{\mathbb R^N} c_n \Bigl\langle  \sum_{j=1}^n   \tilde  Z_{j}, \xi_n \Bigr\rangle\\
 =&
 o(n).
 \end{split}
  \]
This is a contradiction.

\end{proof}

 From now on, we assume that $N\ge 4$.
We want to construct a solution $u$ for \eqref{1.4}  with
\[
u= u_k +\sum_{j=1}^n  U_{p_j,\lambda} +\xi,
\]
where $\xi\in X_s$ is a small perturbed term, satisfying
\[
\sum_{j=1}^n \int_{\mathbb R^N}  U_{p_j}^{p-1}\tilde  Z_{j} \xi=0.
 \]
 Then $\xi$ satisfies
\begin{equation}\label{3-22-12}
  Q_n \xi = l_n+R(\xi),
 \end{equation}
where
\begin{equation}\label{4-22-12}
   l_n =\sum_{j=1}^n \bigl(  V(|y|)-1\bigr)U_{p_j}+  \Bigl(  u_k+\sum_{j=1}^n U_{p_j}\Bigr)^{p} -  u_k^{p} -\sum_{j=1}^n U_{p_j}^{p},
 \end{equation}
and
\begin{equation}\label{5-22-12}
 \begin{split}
   R_n(\xi) = &  \Bigl(  u_k+\sum_{j=1}^n U_{p_j}+\xi\Bigr)^{p}-\Bigl(  u_k+\sum_{j=1}^n U_{p_j}\Bigr)^{p}\\
   &-p \Bigl(  u_k+\sum_{j=1}^n U_{p_j}\Bigr)^{p-1}\xi.
   \end{split}
 \end{equation}

We have the following estimate for $\|l_n\|$.
\begin{lemma}\label{c4-l1-25-3}
There is a small $\sigma>0$, such that
\[
\|l_n\|\le \frac{C }{|p_1|^\alpha}+C e^{- \min[\frac p2-\sigma, 1]|p_2-p_1|},
\]
where $\sigma>0$ is any fixed small constant.

\end{lemma}

\begin{proof}

The proof of this lemma is quite standard. We thus omit it.

\end{proof}

It is also standard to prove the  following lemma.

\begin{lemma}\label{l1-22-12}
\[
\|R_n(\xi)\|\le C\|\xi\|^{\min(p, 2)}.
\]
Moreover,
\[
\|R_n(\xi_1)- R_n(\xi_2)\|\le C \Bigl(  \|\xi_1\|^{\min(p-1, 1)} +\|\xi_2\|^{\min(p-1, 1)}\Bigr)\|\xi_1-\xi_2\|.
\]

\end{lemma}

 We consider the following problem:
\begin{equation}\label{80-3-4}
\begin{cases}
Q_n \xi_n =  l_n + R_n (\xi_n)+ C_N \sum\limits_{j=1}^n   \tilde Z_{j}, \\
\xi_n \in X_s,\\
\displaystyle\sum_{j=1}^n \int_{\mathbb R^N} U_{p_j}^{p-1} \tilde Z_{j} \xi_n=0.
\end{cases}
\end{equation}

Using Lemmas~\ref{l20-2-4}, \ref{c4-l1-25-3} and \ref{l1-22-12}, we can prove the following proposition in a standard way.
\begin{proposition}\label{p1-6-3}
There is an integer $n_0>0$, such that for each $n\ge n_0$  and  $t \in  [t_0 n \ln n, t_1
 n \ln n]$ ,  \eqref{80-3-4} has a solution $\xi_n$ for some constants
$c_n$.   Moreover,  $\xi_n$
is a $C^1$ map from $ [t_0 n \ln n, t_1
 n \ln n]$ to $X_s$,
and
\[
\|\xi_n\|\le  \frac{C }{|p_1|^\alpha}+C e^{- \min[\frac p2-\sigma, 1]|p_2-p_1|}
\]
for any fixed small constant $\sigma>0$.

\end{proposition}

Let
\[
F(t)= I\Bigl( u_{k}+\sum_{j=1}^n U_{p_j}
+\xi_{n}\Bigr).
\]
To obtain a solution of the form $ u_{k}+\sum_{j=1}^n U_{p_j,\lambda}
+\xi_{n}$, we just need to find a critical point for $F(t)$ in $[t_0 n \ln n, t_1
 n \ln n]$.

\begin{proof}[Proof of Theorem~\ref{th11}]

 We have

\begin{equation}\label{3-20-4}
 \begin{split}
 F(t, \lambda)=&   I\Bigl( \sum_{j=1}^n U_{p_j,\lambda}\Bigr) +I( u_k)+ n  O\bigl(  \frac{1 }{t^{2\alpha}}+ e^{- \min[ p-2\sigma, 2]|p_2-p_1|}\bigr)\\
 =& I(u_k)+n A+n\Bigl( \frac{  B_1  }{ t^\alpha}  - (B_2+o(1)) U(|p_2-p_1|)\Bigr)\\
 &+ n  O\bigl(  \frac{1 }{t^{\alpha+1}}+ e^{- \min[ p-2\sigma, 2]|p_2-p_1|}\bigr),
\end{split}
 \end{equation}
  where $A = \frac12\int_{\mathbb R^N}  \bigl(|\nabla U|^2+ U^2\bigr) -\frac1{p+1}\int_{\mathbb R^N} U^{p+1}$, $B_1$ and  $B_2$  are some positive constants,
 and $\sigma>0$ is any fixed small constant.\\

Now  to find a critical point for $F(t)$, we just need to  proceed exactly as in \cite{WY}.
\end{proof}
\bigskip
\noindent\textbf{Acknowledgements} \,\,\,Y. Guo was supported by NNSF of China (No. 11771235). M. Musso was supported by EPSRC research grant EP/T008458/1. S. Peng was supported by NNSF of China (No. 11571130, No. 11831009). S. Yan was supported by NNSF of China (No. 11629101).

\end{document}